\documentclass[11pt]{amsart}
\usepackage[english]{babel}
\usepackage{latexsym}
\usepackage{indentfirst}
\usepackage{amsxtra}
\usepackage{amssymb}
\usepackage{url}
\usepackage{amsmath}
\usepackage{mathrsfs}
\usepackage{hyperref}
\usepackage[T1]{fontenc}
\usepackage[utf8]{inputenc}
\usepackage[all]{xy}
\usepackage{tikz}
\usepackage{graphicx}
\usepackage{tikz-cd}
\usetikzlibrary{arrows.meta}
\usepackage{comment}
\usepackage{cancel}
\usepackage{mathtools}
\usepackage{faktor}
\usepackage{xfrac}
\graphicspath{ {./images/} }
\usepackage{caption}
\usepackage{xcolor}

\newcommand{\R}{\mathbb{R}}

\newcommand{\N}{\mathbb{N}}

\newcommand{\rmn}{\romannumeral}

\newtheorem{lemma}{Lemma}[section]

\newtheorem{thm}[lemma]{Theorem}
\newtheorem*{thm*}{Theorem}
\newtheorem{prop}[lemma]{Proposition}

\newtheorem{defi}[lemma]{Definition}
\newtheorem{rem}[lemma]{Remark}
\newtheorem{ex}[lemma]{Example}

\DeclareMathOperator{\g}{\mathfrak{g}}
\DeclareMathOperator{\n}{\mathfrak{n}}
\DeclareMathOperator{\Ric}{\mathrm{Ric}}
\DeclareMathOperator{\ric}{\mathrm{ric}}
\DeclareMathOperator{\Der}{\mathrm{Der}}
\DeclareMathOperator{\Path}{\mathrm{Path}}
\DeclareMathOperator{\Aut}{\mathrm{Aut}}

\DeclareMathOperator{\id}{id}

\makeatletter
\renewcommand{\section}{\@startsection{section}{1}{\z@}%
  {-3.5ex \@plus -1ex \@minus -.2ex}%
  {2.3ex \@plus.2ex}%
  {\normalfont\centering\bfseries}}
\makeatother

\makeatletter
\renewcommand{\subsection}{\@startsection{subsection}{2}{\z@}%
  {-3.25ex\@plus -1ex \@minus -.2ex}%
  {1.5ex \@plus .2ex}%
  {\normalfont\bfseries}}
\makeatother

\begin{document}
\title
{Nilpotent Lie algebras obtained by ordered sets and Ricci solitons}

\author{Yihao Zheng}
\address{School of Mathematical Sciences, Fudan University, Shanghai 200433, China}
\email{yhzheng24@m.fudan.edu.cn}

\author{Shenglin Zhu}
\address{School of Mathematical Sciences, Fudan University, Shanghai 200433, China}
\email{mazhusl@fudan.edu.cn}


\begin{abstract}
Nilpotent Lie groups with left-invariant metrics provide nontrivial examples of Ricci solitons. Some typical examples are given by the class of two-step nilpotent Lie algebras obtained from simple directed graphs and the class of nilpotent Lie algebras obtained from finite acyclic quivers. In this paper, we generalize the construction of nilpotent Lie algebras that are algebraic Ricci solitons obtained from finite acyclic quivers. We use some special ordered sets to construct nilpotent Lie algebras, which can also be obtained from some special quivers with relations.

A transitively and antisymmetrically ordered set (or TAOS, for short) is a set together with a binary relation that is transitive and antisymmetric. Utilizing the concept of incidence algebras of TAOSs, we construct nilpotent Lie algebras. We modify the method introduced by Mizoguchi and Tamaru \cite{MR4941781} and use it to show that the nilpotent Lie algebras with arbitrarily high degrees of nilpotency obtained from some special finite transitively and antisymmetrically ordered sets, called array TAOSs, are algebraic Ricci solitons. We also give some generalizations of this result, which yield more nilpotent Lie algebras that are algebraic Ricci solitons. Moreover, the corresponding simply-connected nilpotent Lie groups admit left-invariant Ricci solitons.
\end{abstract}

\maketitle

\section{Introduction}
The Riemannian metric $g$ of a Riemannian manifold $(M,g)$ is called a Ricci soliton if its Ricci curvature tensor, denoted by $\ric_g$, satisfies
$$ \ric_g = c \cdot g + \mathcal{L}_X g, $$
for some constant $c \in \mathbb{R}$ and a vector field $X$ on $M$, where $\mathcal{L}_X$ is the Lie derivative with respect to $X$. When $\mathcal{L}_X g = 0$, the metric $g$ is said to be Einstein. The study of Ricci solitons, possibly together with additional structures, remains a central topic in modern geometry. Within this context, homogeneous Ricci solitons have recently attracted significant attention (see, for instance, \cite{MR3305924}, \cite{MR4771217}, \cite{MR2537049}).

The properties of a homogeneous Ricci soliton depend on the sign of $c$. When $c  >   0$, the problem reduces to the study of homogeneous Einstein spaces, since any such soliton decomposes as the direct product of a flat manifold and a homogeneous Einstein manifold with positive scalar curvature (\cite{MR2507581}). Moreover, homogeneous Ricci solitons with $c = 0$ are necessarily flat (\cite{MR362145}). Therefore, nontrivial homogeneous Ricci solitons can only occur when $c   <     0$.

For the case $c   <     0$, B\"{o}hm and Lafuente (\cite{MR4583772}) recently resolved the generalized Alekseevskii conjecture, showing that homogeneous Ricci solitons with $c   <     0$ are isometric to simply-connected solvable Lie groups endowed with left-invariant metrics. A solvable Lie group with a left-invariant metric is called a solvmanifold; similarly, a nilpotent Lie group with a left-invariant metric is called a nilmanifold.

A fundamental structural theorem due to Lauret (\cite{MR2770554}) establishes a correspondence between Ricci soliton solvmanifolds and Ricci soliton nilmanifolds. More precisely, any Ricci soliton nilmanifold can be obtained as the nilradical of a Ricci soliton solvmanifold. Conversely, via suitable solvable extensions, any Ricci soliton solvmanifold can be obtained from a Ricci soliton nilmanifold. This relationship highlights the importance of understanding nilmanifolds for the broader study of homogeneous Ricci solitons.

Two typical classes of Ricci soliton nilmanifolds have been constructed. The first is the class of nilpotent Lie groups obtained from simple directed graphs by Dani and Mainkar (\cite{MR2140439}), which are necessarily two-step nilpotent. The second is the class of nilpotent Lie groups obtained from finite acyclic quivers by Mizoguchi and Tamaru (\cite{MR4941781}), which can have arbitrarily high nilpotency step. Whether the simply-connected nilpotent Lie groups obtained from simple directed graphs admit left-invariant Ricci solitons was determined by Lauret and Will (\cite{MR2785768}), and Mizoguchi and Tamaru (\cite{MR4941781}) showed that all simply-connected nilpotent Lie groups obtained from finite acyclic quivers admit left-invariant Ricci solitons. Since finite acyclic quivers can be used to construct nilpotent Lie groups, it is natural to ask whether this extends to finite acyclic quivers with relations. In this paper, we construct nilpotent Lie algebras from transitively and antisymmetrically ordered sets, whose incidence algebras can be viewed as quiver algebras of special quivers with relations. A transitively and antisymmetrically ordered set (or TAOS, for short) is a set equipped with a binary relation that is transitive and antisymmetric, and it is said to be finite if its cardinality is finite. The notion of the incidence algebra of a TAOS is analogous to the incidence algebra of a transitively ordered set (cf. \cite{MR732192}) or the incidence algebra of a partially ordered set (cf. \cite{MR1442260}). The incidence algebra of a TAOS $P$ is an algebra spanned by the set $\{(s,t)\mid s,t \in P, \ s  \prec    t \} \setminus \{(u,u)\mid u \in P, \ u  \prec    u \}$
whose multiplication is defined as follows: for two basis elements $(s,t)$ and $(u,v)$, the product $(s,t) \cdot (u,v)$ is
\[
(s,t) \cdot (u,v)=
\begin{cases}
(s,v), & \text{if } t = u, \\
0, & \text{if } t \neq u.
\end{cases}
\]
The incidence algebra equipped with the Lie bracket $[x,y] = x \cdot y - y \cdot x$ is the Lie algebra constructed in this paper. The incidence algebra of a finite TAOS is always finite-dimensional and nilpotent. Moreover, it can be nilpotent of arbitrarily high step.

We introduce the concept of array TAOSs, whose definition is as follows.

\begin{defi}
Given $k \in \N$ where $k \geq 2$, $n_1, \dots, n_k \in \N$ where $n_i \geq 1$, $i=1, \dots, k$, $m_1, \dots, m_k \in \N$ where $1 \leq m_i \leq n_i$, $i=1, \dots, k$, and $m_1 = n_1$, $m_k = n_k$, we define the array TAOS $P$ with respect to these data as follows: let $P$ be a set with $\sum_{i=1}^k n_i$ elements, partitioned into $k$ columns where the $i$-th column has cardinality $n_i$, $i=1, \dots, k$. For the $i$-th and $(i+1)$-th columns, choose $m_{i+1}$ elements from the $(i+1)$-th column, and assume that each element of the $i$-th column, denoted by $\eta_i$, and each of the $m_{i+1}$ elements, denoted by $\eta_{i+1}$, satisfy $\eta_i   \prec     \eta_{i+1}$, $i=1, \dots, k-1$. By transitivity, we obtain the required relation on $P$.
\end{defi}

Using this definition, we prove the following theorem.

\begin{thm}\label{Thm:array TAOS Lie group}
The simply-connected Lie groups obtained from a finite array TAOS admit left-invariant Ricci solitons.
\end{thm}

We also give a detailed description of the left-invariant Ricci solitons constructed in our theorem. Moreover, using the concept of algebraic connectivity of quivers (cf. \cite{zheng2026standardpolynomialsprincipalsubalgebras}), we obtain a generalization of Theorem~\ref{Thm:array TAOS Lie group}.

\begin{thm}
Let $P$ be a finite TAOS and $Q$ be its Hasse quiver.  Suppose  all algebraically connected components of $Q$ are Hasse quivers of some array TAOSs. If
for any two distinct algebraically connected components \(Q^{i}\) and
\(Q^{j}\) of \(Q\), there do not exist paths of length at least two \(p_i\) in
\(Q^{i}\) and \(p_j\) in \(Q^{j}\) such that $s(p_i)=s(p_j)$, $t(p_i)=t(p_j)$, 
then the simply-connected Lie group obtained from $P$ admits a left-invariant Ricci soliton.
\end{thm}

We also consider nilpotent Lie algebras obtained from quivers that allow multiple arrows with relations, yielding another generalization of Theorem~\ref{Thm:array TAOS Lie group}.

\begin{thm}
Let $P$ be a finite TAOS, $Q$ be its Hasse quiver.  
Suppose $Q$ satisfies  one of the following: $(\rmn1)$ $Q$ is the Hasse quiver of an array TAOS; $(\rmn2)$  all algebraically connected components of $Q$ are Hasse quivers of some array TAOSs, and 
for any two distinct algebraically connected components \(Q^{i}\) and
\(Q^{j}\) of \(Q\), there do not exist paths of length at least two \(p_i\) in
\(Q^{i}\) and \(p_j\) in \(Q^{j}\) such that $s(p_i)=s(p_j)$, $t(p_i)=t(p_j)$.
Let $Q'$ be the quiver $Q$ with some extra multiple arrows, i.e., $Q'_0 = Q_0$, $Q'_1 = Q_1 \sqcup E$, and for any arrow $\alpha_1 \in E$, there is an arrow $\alpha_2 \in Q_1$ such that $s(\alpha_1) = s(\alpha_2)$ and $t(\alpha_1) = t(\alpha_2)$. Then the simply-connected Lie group $G_{Q'}$ whose corresponding Lie algebra is $\n_{Q'} / I$ admits a left-invariant Ricci soliton, where $I$ is the ideal of $\n_{Q'}$ generated by all differences $\omega_1 - \omega_2$ of two paths $\omega_1$ and $\omega_2$ of length at least two that start at the same vertex and end at the same vertex.
\end{thm}

It is easy to check that the Ricci soliton nilmanifolds constructed in this paper are generally different from those constructed in \cite{MR4941781}; thus we obtain a large family of examples of Ricci soliton nilmanifolds with arbitrarily high degrees of nilpotency.

\section{Preliminaries}

\subsection{Algebraic Ricci solitons}

In this subsection, we recall some preliminaries about simply-connected Lie groups with left-invariant metrics. A Lie algebra $\g$ with an inner product $\langle\cdot,\cdot\rangle$ is called a metric Lie algebra and is denoted by $(\g, \langle\cdot,\cdot\rangle)$. If $(G, g)$ is the simply-connected Lie group with left-invariant metric corresponding to $(\g, \langle\cdot,\cdot\rangle)$, then the curvatures of $(G, g)$ are completely determined by $(\g, \langle\cdot,\cdot\rangle)$.

\begin{defi}
Let $(\g, \langle\cdot,\cdot\rangle)$ be a metric Lie algebra and $X$, $Y\in\g$. The Levi-Civita connection $\nabla:\g\times\g\to\g$ of $(\g, \langle\cdot,\cdot\rangle)$ is defined by
\[
2\langle \nabla_X Y,Z\rangle=\langle[X,Y],Z\rangle+\langle[Z,X],Y\rangle+\langle X,[Z,Y]\rangle.
\]
\end{defi}

We can also write the Levi-Civita connection as
\[
\nabla_X Y=\frac{1}{2}[X,Y]+U(X,Y),
\]
where $U:\g\times\g\to\g$ is a symmetric bilinear form defined by
\[
2\langle U(X,Y),Z\rangle=\langle[Z,X],Y\rangle+\langle X,[Y,Z]\rangle.
\]

\begin{defi}
Let $(\g, \langle\cdot,\cdot\rangle)$ be a metric Lie algebra and $X$, $Y$, $Z \in\g$.

(\romannumeral1) The Riemannian curvature $R: \g \times \g \times \g \to \g$ is defined by
$$R(X, Y)Z \coloneqq \nabla_X \nabla_Y Z - \nabla_Y \nabla_X Z - \nabla_{[X,Y]} Z.$$

(\romannumeral2) The Ricci curvature $\Ric: \g \to \g$ is defined by
$$\Ric(X) \coloneqq \sum_{i=1}^{\dim\g} R(X,e_i)e_i,$$
where $\{e_i\}$ is an orthonormal basis of $\g$.
\end{defi}

We recall some basic notions about Lie algebras.

\begin{defi}
Let $D^0\g = \g$, $D\g = [\g, \g]$, and $D^k\g = D(D^{k-1}\g)$. Then $\g$ is said to be solvable if $D^r\g = 0$ for some $r$.
\end{defi}

\begin{defi}
Let $C^0\g = \g$, and $C^k\g = [C^{k-1}\g, \g]$. Then $\g$ is said to be $n$-step nilpotent if $C^{n-1}\g \neq 0$ but $C^n\g = 0$.
\end{defi}

\begin{defi}
Let $\g$ be a Lie algebra. A linear map $D: \g \to \g$ is called a derivation if
\[
D([a,b]) = [D(a), b] + [a, D(b)]
\]
for all $a, b \in \g$.
Denote by $\Der(\g)$ the set of all derivations of $\g$.
\end{defi}

We now introduce the concept of algebraic Ricci solitons, which is the central topic of this paper.

\begin{defi}
A metric Lie algebra $(\g, \langle\cdot,\cdot\rangle)$ is said to be an algebraic Ricci soliton if there exist $c \in \mathbb{R}$ and $D \in \Der(\g)$ such that
\[
\Ric = c\cdot \id + D.
\]
\end{defi}

For simply-connected nilpotent Lie groups, left-invariant Ricci solitons correspond exactly to algebraic Ricci solitons.

\begin{thm}[Lafuente--Lauret (\cite{MR3230368}), Lauret (\cite{MR1825405})]
Let $(G, g)$ be a simply-connected Lie group with a left-invariant metric, and let $(\g, \langle\cdot,\cdot\rangle)$ be the metric Lie algebra corresponding to $(G, g)$. Then the following hold.

(\romannumeral1) If $(\g, \langle\cdot,\cdot\rangle)$ is an algebraic Ricci soliton, then $(G, g)$ is a Ricci soliton.

(\romannumeral2) If $G$ is nilpotent and $(G, g)$ is a Ricci soliton, then $(\g, \langle\cdot,\cdot\rangle)$ is an algebraic Ricci soliton.
\end{thm}

In view of the generalized Alekseevskii conjecture, resolved in \cite{MR4583772}, the study of Ricci soliton solvmanifolds is of fundamental importance. Moreover, Lauret's well-known work (\cite{MR2770554}) establishes a correspondence between Ricci soliton solvmanifolds and Ricci soliton nilmanifolds.

\begin{thm}[Lauret (\cite{MR2770554})]
\renewcommand{\labelenumi}{(\arabic{enumi})}
(\romannumeral1) Let $(S, g)$ be a solvmanifold, let $(\mathfrak{s}, \langle\cdot,\cdot\rangle)$ be the metric Lie algebra corresponding to $(S, g)$, which is an algebraic Ricci soliton, and let $\n$ be the nilradical of $\mathfrak{s}$. Then the simply-connected nilmanifold corresponding to $(\n, \langle\cdot,\cdot\rangle|_{\n\times\n})$ is a Ricci soliton.

(\romannumeral2) Let $(N, g)$ be a Ricci soliton nilmanifold, and let $(\n, \langle\cdot,\cdot\rangle)$ be the metric Lie algebra corresponding to $(N, g)$. Then there exists a solvable metric Lie algebra $(\mathfrak{s}, \langle\cdot,\cdot\rangle')$ such that the corresponding simply-connected solvmanifold is a Ricci soliton, $\n$ is the nilradical of $\mathfrak{s}$, and $\langle\cdot,\cdot\rangle'|_{\n\times\n} = \langle\cdot,\cdot\rangle$.
\end{thm}

For metric nilpotent Lie algebras, the Ricci curvature satisfies the following formula (cf. \cite{MR2537049}, \cite{MR2792974}), which is more convenient for computation. Let $(\n,\langle\cdot,\cdot\rangle)$ be a metric nilpotent Lie algebra, and let $\{X_1,\dots,X_n\}$ be an orthonormal basis of $(\n,\langle\cdot,\cdot\rangle)$. Then
\begin{equation}\label{Ric_nil}
\begin{split}
\langle\mathrm{Ric}(X),Y\rangle &= -\frac{1}{2}\sum_{i,j}\langle[X,X_i],X_j\rangle\langle[Y,X_i],X_j\rangle \\
&\quad + \frac{1}{4}\sum_{i,j}\langle[X_i,X_j],X\rangle\langle[X_i,X_j],Y\rangle.
\end{split}
\end{equation}

\subsection{Nilpotent Lie algebras obtained by graphs and quivers}

In this subsection, we recall some nilpotent Lie algebras obtained from graphs and quivers.

\begin{defi}
A directed simple graph is an ordered pair $(V,E)$ consisting of a set of vertices $V$ and a set of edges $E \subseteq \{(x, y) \in V \times V \mid x \neq y\}$. For $e = (x, y) \in E$, the vertex $x$ is called the source of $e$ and $y$ is called the target of $e$.

A quiver $Q$ is a quadruple $Q = (Q_0, Q_1, s, t)$ where
\begin{itemize}
\item $Q_0$ is a finite set of vertices,
\item $Q_1$ is a finite set of arrows,
\item $s, t: Q_1 \to Q_0$ are two functions. For an arrow $a \in Q_1$, $s(a)$ and $t(a)$ are called the starting vertex and the terminal vertex of $a$, respectively.
\end{itemize}
\end{defi}

A quiver is a directed graph where loops and multiple arrows between two vertices are allowed. Throughout this paper, we only consider finite acyclic quivers. We refer to \cite{MR3727119} and \cite{MR3308668} for quivers and path algebras.

\begin{defi}
$(\rmn1)$ (Dani--Mainkar \cite{MR2140439}). Let $(V, E)$ be a directed simple graph. The space $\n$ is defined as an $\mathbb{R}$-vector space with basis $V \cup E$. The space $\n$ becomes a $2$-step nilpotent Lie algebra with Lie bracket given by
\[
[x,y]=
\begin{cases}
e, & \text{if } x,y\in V,\ e=(x,y), \\
-e, & \text{if } x,y\in V,\ e=(y,x), \\
0, & \text{otherwise}.
\end{cases}
\]

$(\rmn2)$ (Mizoguchi--Tamaru \cite{MR4941781}). Let $Q = (Q_0, Q_1, s, t)$ be a quiver. The path algebra $\n_Q$ is defined as an $\mathbb{R}$-vector space with basis $\Path(Q)$ and the following product: it is bilinear, and the product of two paths $x$ and $y$ is defined by concatenation,
\[
x\cdot y=
\begin{cases}
xy, & \text{if } xy \text{ is a path}, \\
0, & \text{otherwise}.
\end{cases}
\]
The space $\n_Q$ equipped with the bracket $[x, y] \coloneqq x\cdot y - y\cdot x$ is called the Lie algebra obtained from a quiver.
\end{defi}
Note that path algebras usually include vertices as paths of length zero, but in the above definition, only paths of length $\geq 1$ are considered.

\begin{thm}
$(\rmn1)$ (Lauret--Will \cite{MR2785768}). The nilpotent Lie algebra obtained from a graph admits an algebraic Ricci soliton if and only if the graph is ``positive''.

$(\rmn2)$ (Mizoguchi--Tamaru \cite{MR4941781}). The nilpotent Lie algebra obtained from a finite acyclic quiver admits an algebraic Ricci soliton.
\end{thm}
We refer to \cite{MR2785768} for the definition of the positivity of a graph.

\section{Nilpotent Lie algebras obtained by ordered sets}

In this section, we review some preliminaries on ordered sets and construct Lie algebras from them. We refer to \cite{MR732192} and \cite{MR1442260} for ordered sets and incidence algebras.

\begin{defi}
A transitively and antisymmetrically ordered set $P$ (or TAOS, for short) is a set (which by abuse of notation we also denote by $P$), together with a binary relation denoted $  \prec    $, satisfying the following two axioms:
\begin{itemize}
\item If $s  \prec    t$ and $t  \prec    s$, then $s=t$ (antisymmetry).
\item If $s  \prec    t$ and $t  \prec    u$, then $s  \prec    u$ (transitivity).
\end{itemize}
The binary relation is not necessarily reflexive or irreflexive. Moreover, both partially ordered sets and strictly partially ordered sets are  TAOSs.

$P$ is said to be finite if its cardinality is finite. We use the obvious notation $t  \succ    s$ to mean $s  \prec    t$. Two elements $s$ and $t$ of $P$ are said to be comparable if $s  \prec    t$ or $t  \prec    s$; otherwise, they are incomparable.

A chain in $P$ is a subset $C \subseteq P$ such that for every $s \neq t \in C$, either $s  \prec    t$ or $t  \prec    s$. The length of a chain is its cardinality.
\end{defi}

We refer to \cite{ParteeTerMeulenWall1993} for reflexivity, irreflexivity, partially ordered sets, and strictly partially ordered sets.

\begin{ex}
Let $P= \{ u, v  \}$. Consider the following three binary relations $\prec_1$, $\prec_2$, $\prec_3$ on $P$, where $u \prec_1 u \prec_1 v \prec_1 v$, $u \prec_2 u \prec_2 v$, $u \prec_3 v$. $\prec_1 $ is transitive, antisymmetric and reflexive; $\prec_2$ is transitive and antisymmetric;  $\prec_3 $ is transitive, antisymmetric and irreflexive. So, $(P, \prec_i)$ is a TAOS for each $i$. Moreover, $(P, \prec_1) $ is  a  partially ordered set, and  $(P, \prec_3)$ is  a strictly partially ordered set.

\end{ex}

For a TAOS $P$, we can construct a quiver $Q$ called the Hasse quiver of $P$ as follows.

\begin{defi}
The Hasse quiver $Q$ of a TAOS $P$ is a quadruple $Q = (Q_0, Q_1, s, t)$ where
\begin{itemize}
\item the vertex set $Q_0$ consists of all elements of $P$,
\item the arrow set $Q_1 = \{a \times b \in P \times P \mid a   \prec     b,\ a \neq b,\ \text{and there exists no } t \linebreak[4] \in P \setminus \{a,b\} \text{ such that } a   \prec     t   \prec     b\}$,
\item $s, t: Q_1 \to Q_0$ are two functions. For an arrow $a \times b \in Q_1$, $s(a \times b) = a$ and $t(a \times b) = b$, which are called the starting vertex and the terminal vertex of $a \times b$, respectively.
\end{itemize}
\end{defi}

\begin{ex}
Let $P$ be a TAOS consisting of $n$ integers $1, 2, \dots, n$ satisfying $1   \prec     2   \prec     \dots   \prec     n$. The Hasse quiver of $P$ is
\[\begin{tikzcd}
1 & 2 & \cdots & n
\arrow[from=1-1, to=1-2]
\arrow[from=1-2, to=1-3]
\arrow[from=1-3, to=1-4]
\end{tikzcd}\]
\end{ex}

We can construct a Lie algebra $R(P)$ from a TAOS $P$, called the incidence algebra of $P$, as follows.

\begin{defi}
The incidence algebra $R(P)$ is defined as an $\R$-vector space with basis
\[
B(P) := \{(s,t) \mid s,t \in P,\ s   \prec     t\} \setminus \{(u,u) \mid u \in P,\ u   \prec     u\}
\]
and the following product: it is bilinear, and the product of two basis elements $(s,t)$ and $(u,v)$ is defined by
\[
(s,t) \cdot (u,v) =
\begin{cases}
(s,v), & \text{if } t = u, \\
0, & \text{if } t \neq u.
\end{cases}
\]
The space $R(P)$ equipped with the bracket $[(s,t), (u,v)] \coloneqq (s,t) \cdot (u,v) - (u,v) \cdot (s,t)$ is called the Lie algebra obtained from a TAOS.
\end{defi}

\begin{prop}\label{prop:associative algebra iso}
Let $P$ be a finite TAOS and $Q$ be the Hasse quiver of $P$. Then $R(P)$ is isomorphic to $\n_Q / I$ as Lie algebras, where $\n_Q$ is the path algebra of $Q$ and $I$ is the ideal of $\n_Q$ generated by all differences $\omega_1 - \omega_2$ of two paths $\omega_1$ and $\omega_2$ of length at least two that start at the same vertex and end at the same vertex.
\end{prop}

Throughout this paper, we only consider TAOSs whose longest chain has length at least two, since no nontrivial incidence algebra exists if every chain has length one.

Note that the incidence algebra of $P$ usually contains the basis elements $\{(u,u) \mid u \in P,\ u   \prec     u\}$. However, these elements are not nilpotent, so we only consider the basis $B(P)$ in which every element is nilpotent.

\begin{prop}\label{prop:fd and nil}
Let $P$ be a finite TAOS and let $r$ be the maximal length of chains in $P$. Then $R(P)$ is finite-dimensional and $(r-1)$-step nilpotent.
\end{prop}

\begin{rem}
For convenience, throughout the remaining proofs, we assume that the binary relation is also irreflexive, i.e., $u \prec v$ implies $u \neq v$. The arguments for the non-irreflexive case are similar.
\end{rem}

\begin{proof}[Proof of Proposition~\ref{prop:fd and nil}]
Since $P$ is finite, $B(P)$ must be finite, so $R(P)$ is finite-dimensional.

Let $r$ be the maximal length of chains in $P$. Denote one of the longest chains by $x_1   \prec     x_2   \prec     \cdots   \prec     x_r$; then $\prod_{i=1}^{r-1} (x_i,x_{i+1}) \neq 0$. Moreover, a nonzero product of $r$ basis elements would imply a chain of length $\geq r+1$. Hence $R(P)$ satisfies $C^{r-2}R(P) \neq 0$ and $C^{r-1}R(P) = 0$, which means $R(P)$ is $(r-1)$-step nilpotent.
\end{proof}

\begin{defi}
Let $P$ be a TAOS. An automorphism of $P$ is a bijective function $f: P \to P$ such that $f$ is order-preserving, i.e., $x   \prec     y$ implies $f(x)   \prec     f(y)$, and $f^{-1}$ is also order-preserving.

The automorphism group of $P$ is denoted by $\Aut(P)$.
\end{defi}

Every $f \in \Aut(P)$ acts on $B(P)$ by $f((s,t)) = (f(s), f(t))$, and hence on $R(P)$.

Let $P$ be a finite TAOS and let $r$ be the maximal length of chains in $P$. Denote all chains of length $r$ by
\[
\begin{array}{ccccc}
x_{11} &  \prec    & x_{12} &  \prec     \cdots   \prec    & x_{1r} \\
x_{21} &  \prec    & x_{22} &  \prec     \cdots   \prec    & x_{2r} \\
\vdots && \vdots && \vdots \\
x_{k1} &  \prec    & x_{k2} &  \prec     \cdots   \prec    & x_{kr},
\end{array}
\]
where $\{x_{1,t}, \dots, x_{k,t}\}$ is a multiset (not necessarily a set), $t = 1, \dots, r$. The underlying set of the multiset $\{x_{11}, \dots, x_{k1}\}$ is denoted by $\{a_i\}_{i \in I}$, and the underlying set of $\{x_{12}, \dots, x_{k2}\}$ is denoted by $\{b_j\}_{j \in J}$.

\begin{prop}
Under the above notation, we have the following properties:

(\rmn1) $a_i$ and $a_j$ are incomparable for distinct $i$, $j \in I$.

(\rmn2) $\{a_i\}_{i \in I} \cap \{b_j\}_{j \in J} = \varnothing$.

(\rmn3) $b_i$ and $b_j$ are incomparable for distinct $i$, $j \in J$.
\end{prop}
\begin{proof}
$(\rmn1)$ If $a_i   \prec     a_j$ for some $i$, $j \in I$, we would obtain a chain of length $r+1$, a contradiction.

$(\rmn2)$ If $a_i = b_j$ for some $i \in I$, $j \in J$, then by the definition of $\{b_j\}_{j \in J}$, there must exist $a_k$ with $k \in I$ such that $a_k   \prec     b_j = a_i$, a contradiction.

$(\rmn3)$ If $b_i   \prec     b_j$ for some $i$, $j \in J$, then by the definition of $\{b_j\}_{j \in J}$, there must exist $a_k$ with $k \in I$ such that $a_k   \prec     b_i   \prec     b_j$; hence we obtain a chain of length $r+1$, a contradiction.
\end{proof}

\section{Main theorem}

In this section, we prove the main theorem. Let $P$ be a finite array TAOS. For the nilpotent Lie algebra $R(P)$, we construct an inner product inductively that makes $R(P)$ an algebraic Ricci soliton. The definition of an array TAOS is given later in this section.

\subsection{Nice basis}

In this subsection, we recall the definition of a nice basis and prove that $B(P)$ is a nice basis of $R(P)$.

\begin{defi}[Nikolayevsky (\cite{MR2792974})]\label{nice_def}
Let $\{X_1,\dots,X_n\}$ be a basis of a nilpotent Lie algebra $\n$, with $[X_i,X_j] = \sum_{k=1}^n c_{ij}^k X_k$. Then $\{X_1,\dots,X_n\}$ is said to be {\it nice} if
\begin{itemize}
\item for any $i$ and $j$, there exists at most one $k$ such that $c_{ij}^k \neq 0$,
\item for any $i$ and $k$, there exists at most one $j$ such that $c_{ij}^k \neq 0$.
\end{itemize}
\end{defi}

Nikolayevsky (\cite{MR2792974}) proved that if an orthonormal basis $\{X_1,\dots,X_n\}$ is nice, then the Ricci operator is diagonal with respect to this basis, i.e.,
\[
\langle \Ric(X_i), X_j \rangle = 0,\quad \text{if } i \neq j.
\]
Therefore, by Equation (\ref{Ric_nil}), we obtain the following.

\begin{prop}\label{prop:nice Ricci curvature}
Let $(\n,\langle\cdot,\cdot\rangle)$ be a metric nilpotent Lie algebra, and let $\{X_1,\dots,X_n\}$ be an orthonormal nice basis of $(\n,\langle\cdot,\cdot\rangle)$. Then
\begin{equation}\label{Ric_nice}
\Ric(X_k) = r_k X_k,
\end{equation}
where
$$r_k = -\frac{1}{2}\sum_{i,j}\langle[X_k,X_i],X_j\rangle^2 + \frac{1}{2}\sum_{i <    j}\langle[X_i,X_j],X_k\rangle^2.$$
\end{prop}

We now prove that $B(P)$ is a nice basis of $R(P)$. This simplifies later calculations and also indicates that $R(P)$ has a relatively simple structure.

\begin{prop}
Let $P$ be a finite TAOS, and consider $R(P)$ and $B(P)$ as defined above. Then $B(P)$ is a nice basis of $R(P)$.
\end{prop}
\begin{proof}
We already know that $B(P)$ is a basis of $R(P)$. It is easy to see that $B(P)$ satisfies the first condition of Definition~\ref{nice_def}.

Now take $(\alpha_1, \beta_1), (\alpha_2, \beta_2) \in B(P)$. For any $(\alpha_3, \beta_3) \in B(P)$, we have $[(\alpha_1, \beta_1), (\alpha_3, \beta_3)] = (\alpha_2, \beta_2)$ if and only if $\alpha_1 = \alpha_2$, $\beta_1 = \alpha_3$, $\beta_3 = \beta_2$; and $[(\alpha_1, \beta_1), (\alpha_3, \beta_3)] = -(\alpha_2, \beta_2)$ if and only if $\alpha_3 = \alpha_2$, $\beta_3 = \alpha_1$, $\beta_1 = \beta_2$. By this observation, $B(P)$ satisfies the second condition of Definition~\ref{nice_def}.
\end{proof}

\subsection{Specific subalgebras}

Let $P$ be a finite TAOS, let $r$ be the maximal length of chains in $P$, and consider all chains of length $r$ as above. We can construct a new TAOS $\bar{P}$ from $P$. Set-theoretically, $\bar{P} = P$, but the binary relation is changed: for $i = 1,2,\dots,k$, the elements $x_{i,1}$ and $x_{i,2}$ are made incomparable in $\bar{P}$, while all other relations remain as in $P$.

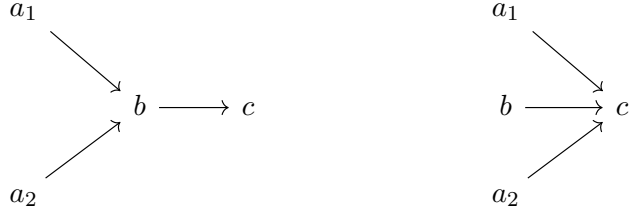
\begin{ex}
\[
\begin{tikzcd}
{a_1} &&&&& {a_1} & \\
& b & c &&& b & c \\
{a_2} &&&&& {a_2}
\arrow[from=1-1, to=2-2]
\arrow[from=1-6, to=2-7]
\arrow[from=2-2, to=2-3]
\arrow[from=2-6, to=2-7]
\arrow[from=3-1, to=2-2]
\arrow[from=3-6, to=2-7]
\end{tikzcd}
\]

We denote the TAOS whose Hasse quiver is on the left by $P_1$, and $\bar{P}_1$ is the TAOS whose Hasse quiver is on the right.
\end{ex}

\begin{prop}
Let $P$ be a finite TAOS, let $r$ be the maximal length of chains in $P$, and consider $\bar{P}$ and all chains of length $r$ as above. Then

(\rmn1) $B(\bar{P})$ is a proper subset of $B(P)$,

(\rmn2) $R(\bar{P})$ is an ideal of $R(P)$,

(\rmn3) the maximal length of chains in $\bar{P}$ is $r-1$; moreover, $R(\bar{P})$ is $(r-2)$-step nilpotent.
\end{prop}

\begin{proof}
$(\rmn1)$ is trivial.

For $(\rmn2)$, take $(\alpha_1,\beta_1) \in B(\bar{P})$ and $(\alpha_2,\beta_2) \in B(P)$. The bracket $[(\alpha_1,\beta_1), \linebreak[4] (\alpha_2,\beta_2)]$ cannot equal $(x_{i1},x_{i2})$ or $-(x_{i1},x_{i2})$ for any $i$, so $[R(\bar{P}),R(P)] \subseteq R(\bar{P})$, which means $R(\bar{P})$ is an ideal of $R(P)$.

$(\rmn3)$ follows from the construction of $\bar{P}$.
\end{proof}

Define $S$ to be $B(P) \setminus B(\bar{P})$; $S$ is nonempty.

Recall that an automorphism $f \in \Aut(P)$ acts on $B(P)$. We prove some properties of automorphisms of TAOSs regarding $S$ and $B(\bar{P})$.

\begin{prop}\label{prop:aut bijective}
Let $P$ be a finite TAOS, let $r$ be the maximal length of chains in $P$, and consider $S$, $\bar{P}$, and all chains of length $r$ as above. Then every $f \in \Aut(P)$ satisfies

(\rmn1) $f(S) = S$,

(\rmn2) $f(B(\bar{P})) = B(\bar{P})$,

(\rmn3) $f|_{B(\bar{P})} \in \Aut(\bar{P})$.
\end{prop}
\begin{proof}
$(\rmn1)$ For each $(x_{i1},x_{i2}) \in S$, there exists a chain $x_{i1}   \prec     x_{i2}   \prec     \cdots   \prec     x_{ir}$ in $P$. Since $f$ is order-preserving, we obtain another chain $f(x_{i1})   \prec     f(x_{i2})   \prec     \cdots   \prec     f(x_{ir})$ in $P$, which means $(f(x_{i1}),f(x_{i2})) \in S$, i.e., $f((x_{i1},x_{i2})) \in S$.

$(\rmn2)$ It is easy to check that $f$ is a bijection on $B(P)$. By $(\rmn1)$, $f(S) = S$, so $f(B(\bar{P})) = B(\bar{P})$.

$(\rmn3)$ By $(\rmn2)$, $f$ is bijective on $\bar{P}$, and both $f$ and $f^{-1}$ are order-preserving; hence $f|_{B(\bar{P})} \in \Aut(\bar{P})$.
\end{proof}

\begin{prop}\label{prop:aut extend}
Let $P$ be a finite TAOS, let $r$ be the maximal length of chains in $P$, and consider $S$, $\bar{P}$, and all chains of length $r$ as above. If for  some $i$  we have $x_{i1}   \prec     \alpha$ in $P$ if and only if $x_{i2}   \prec     \alpha$ in $P$, where $\alpha \neq x_{k1}$ and $\alpha \neq x_{k2}$ for every $k$, then for  $(x_{i1},x_{i2}) \in S$ and  each  $y$ with $(x_{i2}, y) \in B(\bar{P})$, there exists $f \in \Aut(\bar{P})$ such that $f((x_{i2},y)) = (x_{i1},y)$.
\end{prop}
\begin{proof}
Define a map $f: \bar{P} \to \bar{P}$ by
\[
f(x_{i2}) = x_{i1},\qquad f(x_{i1}) = x_{i2},\qquad f|_{\bar{P}\setminus\{x_{i1},x_{i2}\}} = \id.
\]

We show that $f \in \Aut(\bar{P})$. Clearly $f$ and $f^{-1}$ are bijective. If $x_{i1}   \prec     y$ in $\bar{P}$, then $x_{i1}   \prec     y$ in $P$ and $y \neq x_{k1},x_{k2}$ for any $k$. Since $x_{i1}   \prec     \alpha$ in $P$ if and only if $x_{i2}   \prec     \alpha$ in $P$ for $\alpha \neq x_{k1},x_{k2}$, we have $x_{i2}   \prec     y$ in $P$, hence in $\bar{P}$, i.e., $f(x_{i1})   \prec     f(y)$ in $\bar{P}$. If $x_{i2}   \prec     y$ in $\bar{P}$, then $x_{i1}   \prec     y$ in $\bar{P}$, i.e., $f(x_{i2})   \prec     f(y)$ in $\bar{P}$. Thus $f$ is order-preserving, and similarly $f^{-1}$ is order-preserving. Hence $f \in \Aut(\bar{P})$.
\end{proof}

If for each $(x_{i1},x_{i2}) \in S$ we have $x_{i1}   \prec     \alpha$ in $P$ if and only if $x_{i2}   \prec     \alpha$ in $P$, where $\alpha \neq x_{k1},x_{k2}$ for every $k$, then we say that the first column and the second column of $P$ satisfy the exchange condition.

\begin{prop}\label{prop:derivation extend}
Let $P$ be a finite TAOS, $r$ be the maximal length of chains in $P$, and consider $S$, $\bar{P}$, and all chains of length $r$ as above. Assume that the first and second columns of $P$ satisfy the exchange condition, and $D \in \Der(R(\bar{P}))$ satisfies:
\begin{itemize}
\item $D$ is diagonal with respect to the basis $B(\bar{P})$,
\item $D \circ f = f \circ D$ holds for every $f \in \Aut(\bar{P})$.
\end{itemize}
Then the extension $\hat{D}$ of $D$ defined by $\hat{D}|_S = 0$ is a derivation of $R(P)$.
\end{prop}

\begin{proof}
Take $A_1 \coloneqq (\alpha_1,\beta_1)$, $A_2 \coloneqq (\alpha_2,\beta_2) \in B(P)$. We need to show
$$\hat{D}[A_1,A_2] = [\hat{D}A_1, A_2] + [A_1, \hat{D}A_2].$$
If $A_1, A_2 \in B(\bar{P})$, since $D$ is a derivation on $R(\bar{P})$, the claim follows. If $A_1, A_2 \in S$, then $[A_1,A_2] = 0$, so both sides are zero. For $A_1 \in S$ and $A_2 \in B(\bar{P})$, we have $[A_1,A_2] \in R(\bar{P})$, so $\hat{D}[A_1,A_2] = D[A_1,A_2]$. Thus
$$[\hat{D}A_1, A_2] + [A_1, \hat{D}A_2] = [A_1, D A_2],$$
since $\hat{D}A_1 = 0$ and $\hat{D}A_2 = D A_2$.

Since $D$ is diagonal with respect to $B(\bar{P})$, write $D A_2 = k A_2$ with $k \in \R$. If $[A_1,A_2] = 0$, the claim follows. If $[A_1,A_2] \neq 0$, then $\alpha_2 = \beta_1$ and $[A_1,A_2] = A_1 A_2 = (\alpha_1,\beta_2)$. Since the exchange condition holds, by Proposition~\ref{prop:aut extend}, there exists $f \in \Aut(\bar{P})$ such that $f((\alpha_2,\beta_2)) = (\alpha_1,\beta_2)$. Hence
$$D[A_1,A_2] = D((\alpha_1,\beta_2)) = D \circ f((\alpha_2,\beta_2)) = f \circ D((\alpha_2,\beta_2)) = f \circ D A_2.$$
Then $f \circ D A_2 = f(k A_2) = k f(A_2) = k f((\alpha_2,\beta_2)) = k(\alpha_1,\beta_2)$. Moreover, $[A_1, D A_2] = k[A_1,A_2] = k(\alpha_1,\beta_2)$. Thus the claim follows for this case as well.

Therefore, in all cases we have $\hat{D}[A_1,A_2] = [\hat{D}A_1, A_2] + [A_1, \hat{D}A_2]$, completing the proof.
\end{proof}

\subsection{Ricci curvature}

Let $P$ be a finite TAOS, and consider $R(P)$ and $B(P)$ as defined above. Let $\langle\cdot,\cdot\rangle$ be an inner product on $R(P)$ that makes $B(P)$ orthogonal. Denote $\overline{(\alpha, \beta)} \coloneqq (\alpha, \beta) / \left|(\alpha, \beta)\right|$; then $\{ \overline{(\alpha, \beta)} \mid (\alpha, \beta) \in B(P) \}$ is an orthonormal basis and, moreover, is also nice. Thus, by Proposition~\ref{prop:nice Ricci curvature}, we obtain

\begin{align}\label{TAOS Ricci curvature}
\Ric(\overline{(y_1,y_2)}) &=
\Bigl(-\frac{1}{2}\sum_{z  \prec    y_1} \langle [\overline{(z,y_1)},\overline{(y_1,y_2)}], \overline{(z,y_2)} \rangle^2 \nonumber \\
&\qquad -\frac{1}{2}\sum_{z  \succ    y_2} \langle [\overline{(y_1,y_2)},\overline{(y_2,z)}], \overline{(y_1,z)} \rangle^2 \\
&\qquad +\frac{1}{2}\sum_{y_1  \prec    z  \prec    y_2} \langle [\overline{(y_1,z)},\overline{(z,y_2)}], \overline{(y_1,y_2)} \rangle^2 \Bigr) \overline{(y_1,y_2)}. \nonumber
\end{align}

We keep the notation $R(P)$, $B(P)$, $S$, $\bar{P}$, $\{a_i\}_{i \in I}$, and $\{b_j\}_{j \in J}$ as defined above, and give an expression for $\Ric_{R(P)}(\overline{(\alpha, \beta)})$ in terms of the Ricci curvature $\Ric_{R(\bar{P})}$ of $R(\bar{P})$ with respect to $\langle\cdot,\cdot\rangle|_{R(\bar{P}) \times R(\bar{P})}$.

\begin{thm}\label{Thm:Ricci curvature}
Under the above notation, we have the following formulae.

(\rmn1) For each $(a_i,b_j) \in S$,
$$\Ric_{R(P)}(\overline{(a_i,b_j)}) = -\frac{1}{2}\sum_{\{y\, \mid \, y \succ b_j  \}} \langle [\overline{(a_i,b_j)},\overline{(b_j,y)}], \overline{(a_i,y)} \rangle^2 \; \overline{(a_i,b_j)}.$$

(\rmn2) For each $(a_i,y) \in B(P)$ with $y \neq b_j$ for all $j$,
\begin{align*}
\Ric_{R(P)}(\overline{(a_i,y)}) &= \Ric_{R(\bar{P})}(\overline{(a_i,y)}) \\
&\quad + \frac{1}{2} \sum_{\substack{%k \\ a_i   \prec     b_k   \prec     y}
\{k\, \mid \, a_i   \prec     b_k   \prec     y\}}} \langle [\overline{(a_i,b_k)},\overline{(b_k,y)}], \overline{(a_i,y)} \rangle^2 \; \overline{(a_i,y)}.
\end{align*}

(\rmn3) For each $(b_j,y) \in B(P)$,
\begin{align*}
\Ric_{R(P)}(\overline{(b_j,y)}) &= \Ric_{R(\bar{P})}(\overline{(b_j,y)}) \\
&\quad -\frac{1}{2} \sum_{%\substack{k \\ a_k   \prec     b_j   \prec     y}
\{k\, \mid \, a_k   \prec     b_j   \prec     y\}} \langle [\overline{(a_k,b_j)},\overline{(b_j,y)}], \overline{(a_k,y)} \rangle^2 \; \overline{(b_j,y)}.
\end{align*}

(\rmn4) For each $(y_1,y_2) \in B(P)$ with $y_1, y_2 \neq a_i, b_j$ for any $i, j$,
$$\Ric_{R(P)}(\overline{(y_1,y_2)}) = \Ric_{R(\bar{P})}(\overline{(y_1,y_2)}).$$
\end{thm}

\begin{proof}
We prove these formulae using (\ref{TAOS Ricci curvature}).

For $(\rmn1)$, there is no $\alpha \in P$ with $\alpha   \prec     a_i$ nor with $a_i   \prec     \alpha   \prec     b_j$. Hence the first and third sums in (\ref{TAOS Ricci curvature}) vanish, yielding the formula.

For $(\rmn2)$, there is no $\alpha \in P$ with $\alpha   \prec     a_i$, so the first sum in (\ref{TAOS Ricci curvature}) vanishes. The second sum is already contained in $\Ric_{R(\bar{P})}(\overline{(a_i,y)})$. A term in the third sum is not contained in $\Ric_{R(\bar{P})}(\overline{(a_i,y)})$ only if $z = b_k$ for some $k$. Hence the formula follows.

For $(\rmn3)$, the second and third sums in (\ref{TAOS Ricci curvature}) are contained in $\Ric_{R(\bar{P})}(\overline{(b_j,y)})$, and $z   \prec     b_j$ occurs only when $z = a_k$ for some $k$.

For $(\rmn4)$, all sums in (\ref{TAOS Ricci curvature}) are already contained in $\Ric_{R(\bar{P})}(\overline{(y_1,y_2)})$.
\end{proof}

The set of all $(a_i,y) \in B(P)$ with $y \neq b_j$ for all $j$ is denoted $B_1$; the set of all $(b_j,y) \in B(P)$ is denoted $B_2$; the set of all $(y_1,y_2) \in B(P)$ with $y_1, y_2 \neq a_i, b_j$ for any $i, j$ is denoted $B_3$. Note that $B(P) = S \sqcup B_1 \sqcup B_2 \sqcup B_3$.

\subsection{Proof of the main theorem}

In this subsection, we give the definition of an array TAOS and prove our main theorem.

\begin{defi}
Given $k \in \N$ where $k \geq 2$, $n_1, \dots, n_k \in \N$ where $n_i \geq 1$, $i=1,\dots,k$, $m_1,\dots,m_k \in \N$ where $1 \leq m_i \leq n_i$, $i=1,\dots,k$, and $m_1 = n_1$, $m_k = n_k$, we define the array TAOS $P$ with respect to these data as follows: let $P$ be a set with $\sum_{i=1}^k n_i$ elements, partitioned into $k$ columns where the $i$-th column has cardinality $n_i$, $i=1,\dots,k$. For the $i$-th and $(i+1)$-th columns, choose $m_{i+1}$ elements from the $(i+1)$-th column, and assume that each element of the $i$-th column, denoted by $\eta_i$, and each of the $m_{i+1}$ elements, denoted by $\eta_{i+1}$, satisfy $\eta_i   \prec     \eta_{i+1}$, $i=1,\dots,k-1$. By transitivity, we obtain the required relation on $P$.
\end{defi}

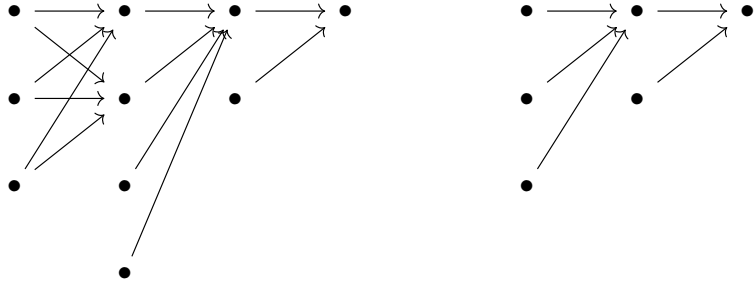
\begin{ex}
\[\begin{tikzcd}
\bullet & \bullet & \bullet & \bullet && \bullet & \bullet & \bullet \\
\bullet & \bullet & \bullet &&& \bullet & \bullet \\
\bullet & \bullet &&&& \bullet \\
& \bullet
\arrow[from=1-1, to=1-2]
\arrow[from=1-1, to=2-2]
\arrow[from=1-2, to=1-3]
\arrow[from=1-3, to=1-4]
\arrow[from=1-6, to=1-7]
\arrow[from=1-7, to=1-8]
\arrow[from=2-1, to=1-2]
\arrow[from=2-1, to=2-2]
\arrow[from=2-2, to=1-3]
\arrow[from=2-3, to=1-4]
\arrow[from=2-6, to=1-7]
\arrow[from=2-7, to=1-8]
\arrow[from=3-1, to=1-2]
\arrow[from=3-1, to=2-2]
\arrow[from=3-2, to=1-3]
\arrow[from=3-6, to=1-7]
\arrow[from=4-2, to=1-3]
\end{tikzcd}\]
The Hasse quiver of the array TAOS with parameters $n=(3,4,2,1)$, $m=(3,2,1,1)$ is the quiver on the left, and the Hasse quiver of the array TAOS with parameters $n=(3,2,1)$, $m=(3,1,1)$ is the quiver on the right.
\end{ex}

We can check that for an array TAOS $P$, the first and second columns of $P$ satisfy the exchange condition. Moreover, $\bar{P}$ is still an array TAOS, and its first and second columns also satisfy the exchange condition. Hence we can use the propositions in the previous subsections in the inductive proof of the following main theorem.

\begin{thm}\label{Thm:Ricci soliton}
Let $P$ be a finite array TAOS, and let $R(P)$ be the nilpotent Lie algebra obtained from $P$. Then there exists an inner product $\langle\cdot,\cdot\rangle$ on $R(P)$ satisfying:

(\rmn1) $\langle\cdot,\cdot\rangle$ is an algebraic Ricci soliton with $\Ric = -\id + D$, where $D \in \Der(R(P))$ is diagonal with respect to $B(P)$,

(\rmn2) $B(P)$ is orthogonal with respect to $\langle\cdot,\cdot\rangle$,

(\rmn3) $\langle\cdot,\cdot\rangle$ is preserved by $\Aut(P)$.
\end{thm}

\begin{proof}
Let $l$ be the maximal length of chains in $P$. We prove the theorem by induction on $l$. If $l = 2$, then $R(P)$ is $1$-step nilpotent, hence abelian. Let $\langle\cdot,\cdot\rangle_1$ be an inner product on $R(P)$ that makes $B(P)$ orthonormal. Then $(R(P), \langle\cdot,\cdot\rangle_1)$ satisfies $(\rmn2)$ and $(\rmn3)$. Since $R(P)$ is abelian, $\Ric = 0 = -\id + \id$, where $\id$ is a derivation of $R(P)$; hence $(\rmn1)$ holds.

Assume the assertion holds for $l = r-1$ with $r \geq 3$, and prove it for $l = r$. Let $R(P)$ be the nilpotent Lie algebra obtained from an array TAOS $P$ whose maximal chain length is $r$, and assume the $i$-th column of $P$ has $n_i$ elements of which $m_i$ are chosen, $i=1,\dots,r$.

By the induction hypothesis, there exists an inner product $\langle\cdot,\cdot\rangle'$ on $R(\bar{P})$ satisfying the three conditions. We define an inner product $\langle\cdot,\cdot\rangle$ on $R(P)$ as follows:

\begin{itemize}
\item $\langle\cdot,\cdot\rangle|_{R(\bar{P}) \times R(\bar{P})} = \langle\cdot,\cdot\rangle'$,
\item $B(P)$ is orthogonal with respect to $\langle\cdot,\cdot\rangle$,
\item $\langle (x_{i1},x_{i2}), (x_{i1},x_{i2}) \rangle = \dfrac{\sum_{t=1}^r m_t}{2}$ for any $i$.
\end{itemize}

Since $\langle (x_{i1},x_{i2}), (x_{i1},x_{i2}) \rangle   >     0$ for each $i$, this defines a genuine inner product. By definition, $(\rmn2)$ is satisfied.

Take any $f \in \Aut(P)$ and $(\alpha_1,\beta_1) \in B(P)$. If $(\alpha_1,\beta_1) \in B(\bar{P})$, then by Proposition~\ref{prop:aut bijective}, $f|_{B(\bar{P})} \in \Aut(\bar{P})$, and by the induction hypothesis,
\begin{align*}
\langle (\alpha_1,\beta_1), (\alpha_1,\beta_1) \rangle
&= \langle (\alpha_1,\beta_1), (\alpha_1,\beta_1) \rangle' \\
&= \langle f|_{B(\bar{P})}((\alpha_1,\beta_1)), f|_{B(\bar{P})}((\alpha_1,\beta_1)) \rangle' \\
&= \langle f((\alpha_1,\beta_1)), f((\alpha_1,\beta_1)) \rangle.
\end{align*}
If $(\alpha_1,\beta_1) \in S$, then $f((\alpha_1,\beta_1)) \in S$, so
$$\langle (\alpha_1,\beta_1), (\alpha_1,\beta_1) \rangle = \langle f((\alpha_1,\beta_1)), f((\alpha_1,\beta_1)) \rangle = \frac{\sum_{t=1}^r m_t}{2}.$$
Thus $(\rmn3)$ holds.

For $(\rmn1)$, we first prove the claim that for any $y   \succ     x_{i2}$,
\begin{align}\label{inner product}
\langle [\overline{(x_{i1},x_{i2})}, \overline{(x_{i2},y)}], \overline{(x_{i1},y)} \rangle = \frac{1}{\left|(x_{i1},x_{i2})\right|}.
\end{align}

Indeed,
\begin{align*}
\langle [\overline{(x_{i1},x_{i2})}, \overline{(x_{i2},y)}], \overline{(x_{i1},y)} \rangle
&= \frac{\langle [(x_{i1},x_{i2}),(x_{i2},y)], (x_{i1},y) \rangle}
        {\left|(x_{i1},x_{i2})\right|\;\left|(x_{i2},y)\right|\;\left|(x_{i1},y)\right|} \\
&= \frac{\langle (x_{i1},y), (x_{i1},y) \rangle}
        {\left|(x_{i1},x_{i2})\right|\;\left|(x_{i2},y)\right|\;\left|(x_{i1},y)\right|} \\
&= \frac{\left|(x_{i1},y)\right|}{\left|(x_{i1},x_{i2})\right|\;\left|(x_{i2},y)\right|}.
\end{align*}

By Proposition~\ref{prop:aut extend}, there exists $f \in \Aut(\bar{P})$ such that $f((x_{i2},y)) = (x_{i1},y)$. By the induction hypothesis, $f$ preserves $\langle\cdot,\cdot\rangle'$, hence
\begin{align*}
\langle (x_{i1},y), (x_{i1},y) \rangle
&= \langle (x_{i1},y), (x_{i1},y) \rangle' \\
&= \langle (x_{i2},y), (x_{i2},y) \rangle' \\
&= \langle (x_{i2},y), (x_{i2},y) \rangle,
\end{align*}
so $\left|(x_{i1},y)\right| = \left|(x_{i2},y)\right|$. The claim follows.

Using this claim, we rewrite the formulae in Theorem~\ref{Thm:Ricci curvature}.
\begin{itemize}
\item For each $(a_i,b_j) \in S$,
$$\Ric_{R(P)}(\overline{(a_i,b_j)}) = -\frac{1}{2}\sum_{\{y\, \mid \, y \succ b_j  \}} \frac{1}{\left|(a_i,b_j)\right|^2} \; \overline{(a_i,b_j)}.$$

\item For each $(a_i,y) \in B_1$,
$$\Ric_{R(P)}(\overline{(a_i,y)}) = \Ric_{R(\bar{P})}(\overline{(a_i,y)}) + \frac{1}{2} \sum_{\{k\, \mid \, a_i   \prec     b_k   \prec     y\}} \frac{1}{\left|(a_i,b_k)\right|^2} \; \overline{(a_i,y)}.$$

\item For each $(b_j,y) \in B_2$,
$$\Ric_{R(P)}(\overline{(b_j,y)}) = \Ric_{R(\bar{P})}(\overline{(b_j,y)}) - \frac{1}{2} \sum_{\{k\, \mid \, a_k   \prec     b_j   \prec     y\}} \frac{1}{\left|(a_k,b_j)\right|^2} \; \overline{(b_j,y)}.$$

\item For each $(y_1,y_2) \in B_3$,
$$\Ric_{R(P)}(\overline{(y_1,y_2)}) = \Ric_{R(\bar{P})}(\overline{(y_1,y_2)}).$$
\end{itemize}

Now we show $\Ric_{R(P)} = -\id + D$. By the induction hypothesis, $\Ric_{R(\bar{P})} = -\id + H$, where $H \in \Der(R(\bar{P}))$ is diagonal with respect to $B(\bar{P})$. Hence we can write
$$\Ric_{R(P)} = -\id + \hat{H} + A,$$
where $\hat{H}$ is the extension of $H$ as in Proposition~\ref{prop:derivation extend}, and
\[
A(x) =
\begin{cases}
\bigl(-\frac{1}{2}\sum_{\{y\, \mid \, y \succ b_j  \}} \frac{1}{\left|(a_i,b_j)\right|^2} + 1\bigr)x, & \text{if } x = (a_i,b_j) \in S, \\[4pt]
\bigl(\frac{1}{2} \sum_{\{k\, \mid \, a_i   \prec     b_k   \prec     y\}} \frac{1}{\left|(a_i,b_k)\right|^2}\bigr)x, & \text{if } x = (a_i,y) \in B_1, \\[4pt]
\bigl(-\frac{1}{2} \sum_{\{k\, \mid \, a_k   \prec     b_j   \prec     y\}} \frac{1}{\left|(a_k,b_j)\right|^2}\bigr)x, & \text{if } x = (b_j,y) \in B_2, \\[4pt]
0, & \text{if } x = (y_1,y_2) \in B_3.
\end{cases}
\]
It is easy to see that $\hat{H}+A$ is diagonal with respect to $B(P)$. For every $f \in \Aut(\bar{P})$, we have $f([x,y]) = [f(x), f(y)]$ and $\langle x, x \rangle' = \langle f(x), f(x) \rangle'$, which implies that $f$ commutes with $\Ric_{R(\bar{P})}$, hence with $H = \Ric_{R(\bar{P})} + \id_{R(\bar{P})}$. Therefore, by Proposition~\ref{prop:derivation extend}, $\hat{H} \in \Der(R(P))$.

It remains to show $A \in \Der(R(P))$, i.e., for all $x, y \in B(P)$,
$$A[x, y] = [Ax, y] + [x, Ay].$$
This holds for $x, y \in B(P)$ with $[x, y] = 0$, since $A$ is diagonal. The nonzero brackets are:

$(1)$ $[(y_1,y_2), (y_2,y_3)] = (y_1,y_3)$ with $(y_1,y_2), (y_2,y_3) \in B_3$,

$(2)$ $[(a_i,y_1), (y_1,y_2)] = (a_i,y_2)$ with $(a_i,y_1) \in B_1$, $(y_1,y_2) \in B_3$, \linebreak[4]
 $[(b_j,y_1), (y_1,y_2)] = (b_j,y_2)$ with $(b_j,y_1) \in B_2$, $(y_1,y_2) \in B_3$,

$(3)$ $[(a_i,b_j), (b_j,y)] = (a_i,y)$ with $(a_i,b_j) \in S$, $(b_j,y) \in B_2$.

Case $(1)$ is trivial. For case $(2)$, by the structure of $P$, the sets $\{k \in J \mid a_i   \prec     b_k   \prec     y_1\}$ and $\{k \in J \mid a_i   \prec     b_k   \prec     y_2\}$ are equal for each $i \in I$, and similarly $\{k \in I \mid a_k   \prec     b_j   \prec     y_1\} = \{k \in I \mid a_k   \prec     b_j   \prec     y_2\}$ for each $j \in J$. Hence the equality holds. For case $(3)$, the equality reduces to
$$\frac{1}{2}\sum_{\{y\, \mid \, y \succ b_j  \}} \frac{1}{\left|(a_i,b_j)\right|^2}
+ \frac{1}{2} \sum_{\{k\, \mid \, a_i   \prec     b_k   \prec     y\}} \frac{1}{\left|(a_i,b_k)\right|^2}
+ \frac{1}{2} \sum_{\{k\, \mid \, a_k   \prec     b_j   \prec     y\}} \frac{1}{\left|(a_k,b_j)\right|^2}
= 1.$$
By the structure of $P$ and the definition of $\langle\cdot,\cdot\rangle$,
\begin{align*}
&\; \frac{1}{2}\sum_{\{y\, \mid \, y \succ b_j  \}} \frac{1}{\left|(a_i,b_j)\right|^2}
   + \frac{1}{2} \sum_{\{k\, \mid \, a_i   \prec     b_k   \prec     y\}} \frac{1}{\left|(a_i,b_k)\right|^2}
   + \frac{1}{2} \sum_{\{k\, \mid \, a_k   \prec     b_j   \prec     y\}} \frac{1}{\left|(a_k,b_j)\right|^2} \\
&= \frac{1}{2}\Bigl(\sum_{t=3}^r m_t\Bigr) \frac{2}{\sum_{t=1}^r m_t}
   + \frac{1}{2} m_2 \frac{2}{\sum_{t=1}^r m_t}
   + \frac{1}{2} m_1 \frac{2}{\sum_{t=1}^r m_t} \\
&= 1.
\end{align*}
Hence $A \in \Der(R(P))$, proving $(\rmn1)$ and completing the proof.
\end{proof}

We now give a more detailed description of the inner product obtained in the above theorem.

\begin{prop}\label{prop:same inner product}
Let $P$ be a finite array TAOS, let $R(P)$ be the nilpotent Lie algebra obtained from $P$, and let $\langle\cdot,\cdot\rangle$ be the inner product from Theorem~\ref{Thm:Ricci soliton}. Let $x, y, z \in P$ satisfy $x   \prec     y   \prec     z$, and suppose there is no $t \in P$ such that $y   \prec     t   \prec     z$. Then $\left|(x,z)\right| = \left|(y,z)\right|$ with respect to $\langle\cdot,\cdot\rangle$.
\end{prop}

\begin{proof}
Denote $\bar{P}$ by $F(P)$. Let $r$ be the maximal length of chains in $P$. Then $2$ is the maximal length of chains in $F^{(r-2)}(P)$. We extend the inner product on $R(F^{(r-2)}(P))$ step by step to obtain the inner product $\langle\cdot,\cdot\rangle$ on $R(P)$. At each inductive step, we obtain a set $S = B(F^{i}(P)) \setminus B(F^{(i+1)}(P))$. Throughout this procedure, $(x,z)$ and $(y,z)$ belong to the same set $S$ obtained at the same inductive step. By the definition of $\langle\cdot,\cdot\rangle$ in Theorem~\ref{Thm:Ricci soliton}, we have $\langle (x,z), (x,z) \rangle = \langle (y,z), (y,z) \rangle$, i.e., $\left|(x,z)\right| = \left|(y,z)\right|$.
\end{proof}

Using this proposition, we can easily compute the norm of any element of $B(P)$.

\section{Generalizations}

In this section, we give some generalizations of our main theorem.

First, we make the following observation.

\begin{prop}\label{prop:direct sum}
Let $(\g_1, \langle\cdot,\cdot\rangle_1)$ and $(\g_2, \langle\cdot,\cdot\rangle_2)$ be two metric Lie algebras that are algebraic Ricci solitons with respect to the same constant $c$. Then there exists an inner product $\langle\cdot,\cdot\rangle$ on $\g_1 \oplus \g_2$ such that $(\g_1 \oplus \g_2, \langle\cdot,\cdot\rangle)$ is an algebraic Ricci soliton with respect to $c$. Moreover, if $(\g_1, \langle\cdot,\cdot\rangle_1), \dots, (\g_k, \langle\cdot,\cdot\rangle_k)$ are metric Lie algebras that are algebraic Ricci solitons with respect to the same constant $c$, then there exists an inner product $\langle\cdot,\cdot\rangle$ on $\bigoplus_{i=1}^k \g_i$ such that $(\bigoplus_{i=1}^k \g_i, \langle\cdot,\cdot\rangle)$ is an algebraic Ricci soliton with respect to $c$.
\end{prop}

We introduce the notion of algebraic connectivity of quivers, for details, see \cite{zheng2026standardpolynomialsprincipalsubalgebras}.

\begin{defi}
A finite quiver $Q$ is said to be algebraically connected if $\n_Q$ is not isomorphic to a direct sum of more than one path algebra of a quiver as associative algebras.

If a finite quiver $Q$ is not algebraically connected, since $Q$ is finite, we can always write $\n_Q \cong \bigoplus_{i=1}^n \n_{E_i}$, where the $E_i$ are finite algebraically connected subquivers of $Q$. These $E_i$ are called the algebraically connected components of $Q$.
\end{defi}

\begin{thm}\label{Thm:alg conected}
Let $P$ be a finite TAOS and let $Q$ be its Hasse quiver. Suppose  all algebraically connected components of $Q$ are Hasse quivers of some array TAOSs. If 
for any two distinct algebraically connected components \(Q^{i}\) and
\(Q^{j}\) of \(Q\), there do not exist paths of length at least two \(p_i\) in
\(Q^{i}\) and \(p_j\) in \(Q^{j}\) such that $s(p_i)=s(p_j)$, $t(p_i)=t(p_j)$,  
then there exists an inner product $\langle\cdot,\cdot\rangle$ on $R(P)$ such that $(R(P), \langle\cdot,\cdot\rangle)$ is an algebraic Ricci soliton with respect to $-1$.

\end{thm}

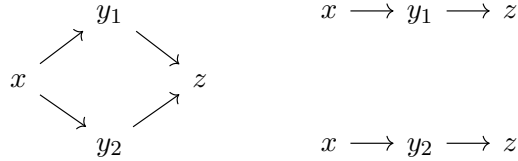
\begin{rem}
\[\begin{tikzcd}[row sep=1.0em, column sep=1.5em]
  & {y_1} &&& x & {y_1} & z \\
  x && z &&&& {} \\
  & {y_2} &&& x & {y_2} & z
  \arrow[from=1-2, to=2-3]
  \arrow[from=1-5, to=1-6]
  \arrow[from=1-6, to=1-7]
  \arrow[from=2-1, to=1-2]
  \arrow[from=2-1, to=3-2]
  \arrow[from=3-2, to=2-3]
  \arrow[from=3-5, to=3-6]
  \arrow[from=3-6, to=3-7]
\end{tikzcd}\]

Let $P$ be a finite TAOS and let $Q$ be its Hasse quiver. The fact that $P$ is a finite array TAOS does not  imply that all algebraically connected components of $Q$ are Hasse quivers of some array TAOSs, and 
for any two distinct algebraically connected components \(Q^{i}\) and
\(Q^{j}\) of \(Q\), there do not exist paths of length at least two \(p_i\) in
\(Q^{i}\) and \(p_j\) in \(Q^{j}\) such that $s(p_i)=s(p_j)$, $t(p_i)=t(p_j)$.
For example,  see the quiver $\hat{Q}$ on the left of the above figure. It is   the Hasse quiver of the array TAOS with parameters $n=m=(1,2,1)$, and its algebraically connected components $\hat{Q}^1$ and $\hat{Q}^2$ are on the right. We can see that both $\hat{Q}^1$ and $\hat{Q}^2$ contain a path of length two from $x$ to $z$.
\end{rem}

\begin{proof}[Proof of Theorem~\ref{Thm:alg conected}]
By the definition of algebraically connected components and the hypothesis, $R(P) \cong \bigoplus_i R(P_i)$, where the $P_i$ are the TAOSs whose Hasse quivers are the algebraically connected components of $Q$. Moreover, each $P_i$ is an array TAOS. By Theorem~\ref{Thm:Ricci soliton}, for each $R(P_i)$ there exists an inner product $\langle\cdot,\cdot\rangle_i$ such that $(R(P_i), \langle\cdot,\cdot\rangle_i)$ is an algebraic Ricci soliton with respect to $-1$. The result then follows from Proposition~\ref{prop:direct sum}.
\end{proof}

\begin{ex}
\[
\begin{tikzcd}[row sep=1.15em, column sep=1.5em,
  cells={nodes={inner sep=1pt}}]
& \bullet & \bullet & \bullet & \bullet & \bullet & \bullet & \\
& \bullet & \bullet & \bullet && \bullet & \bullet \\
&& \bullet &&& \bullet & \bullet \\
&&&&&& \bullet \\
\\[-1.0em]
\bullet & \bullet & \bullet & \bullet && \bullet & \bullet & \bullet \\
\bullet & \bullet & \bullet &&& \bullet & \bullet \\
& \bullet &&&& \bullet & \bullet \\
&&&&& \bullet
\arrow[from=1-2, to=1-3]
\arrow[from=1-2, to=2-3]
\arrow[from=1-3, to=1-4]
\arrow[from=1-3, to=2-4]
\arrow[from=1-4, to=1-5]
\arrow[from=1-6, to=1-5]
\arrow[from=1-7, to=1-6]
\arrow[from=1-7, to=2-6]
\arrow[from=2-2, to=1-3]
\arrow[from=2-2, to=2-3]
\arrow[from=2-3, to=1-4]
\arrow[from=2-3, to=2-4]
\arrow[from=2-4, to=1-5]
\arrow[from=2-6, to=1-5]
\arrow[from=2-7, to=1-6]
\arrow[from=2-7, to=2-6]
\arrow[from=3-3, to=1-4]
\arrow[from=3-3, to=2-4]
\arrow[from=3-6, to=1-5]
\arrow[from=3-7, to=1-6]
\arrow[from=3-7, to=2-6]
\arrow[from=4-7, to=1-6]
\arrow[from=4-7, to=2-6]
\arrow[from=6-1, to=6-2]
\arrow[from=6-1, to=7-2]
\arrow[from=6-2, to=6-3]
\arrow[from=6-2, to=7-3]
\arrow[from=6-3, to=6-4]
\arrow[from=6-6, to=6-7]
\arrow[from=6-6, to=7-7]
\arrow[from=6-7, to=6-8]
\arrow[from=7-1, to=6-2]
\arrow[from=7-1, to=7-2]
\arrow[from=7-2, to=6-3]
\arrow[from=7-2, to=7-3]
\arrow[from=7-3, to=6-4]
\arrow[from=7-6, to=6-7]
\arrow[from=7-6, to=7-7]
\arrow[from=7-7, to=6-8]
\arrow[from=8-2, to=6-3]
\arrow[from=8-2, to=7-3]
\arrow[from=8-6, to=6-7]
\arrow[from=8-6, to=7-7]
\arrow[from=8-7, to=6-8]
\arrow[from=9-6, to=6-7]
\arrow[from=9-6, to=7-7]
\end{tikzcd}\]

There exists an inner product $\langle\cdot,\cdot\rangle$ on the incidence algebra $R(P)$ of the TAOS $P$ whose Hasse quiver $Q$ is the quiver in the upper part of the figure, such that $(R(P), \langle\cdot,\cdot\rangle)$ is an algebraic Ricci soliton. This holds because the algebraically connected components of $Q$ are the two quivers $Q^1$, $Q^2$ in the lower part, which are Hasse quivers of array TAOSs, and there do not exist paths of length at least two \(p_1\) in
\(Q^{1}\) and \(p_2\) in \(Q^{2}\) such that $s(p_1)=s(p_2)$, $t(p_1)=t(p_2)$.
\end{ex}

The Hasse quiver of a TAOS is necessarily a $0$--$1$ quiver, meaning that for any two vertices $a$ and $b$, there is at most one arrow from $a$ to $b$. We can also consider nilpotent Lie algebras obtained from quivers that allow multiple arrows with relations.

We need the following observation to obtain our generalization about nilpotent Lie algebras obtained from quivers with multiple arrows and relations.

\begin{prop}\label{prop:abelian direct sum}
Let $(\g_1, \langle\cdot,\cdot\rangle_1)$ be a metric Lie algebra that is an algebraic Ricci soliton with respect to $c$, and let $\g_2$ be an abelian Lie algebra. Then there exists an inner product $\langle\cdot,\cdot\rangle$ on $\g_1 \oplus \g_2$ such that $(\g_1 \oplus \g_2, \langle\cdot,\cdot\rangle)$ is an algebraic Ricci soliton with respect to $c$.
\end{prop}

\begin{proof}
For any inner product $\langle\cdot,\cdot\rangle_2$ on $\g_2$, we have $-c \cdot \id|_{\g_2} \in \Der(\g_2)$ and $0 = \Ric_{\g_2} = c \cdot \id|_{\g_2} + (-c \cdot \id|_{\g_2})$, so $(\g_2, \langle\cdot,\cdot\rangle_2)$ is an algebraic Ricci soliton with respect to $c$. The result follows from Proposition~\ref{prop:direct sum}.
\end{proof}

\begin{thm}\label{Thm:multiedge}
Let $P$ be a finite TAOS, let $Q$ be its Hasse quiver. Suppose $Q$ satisfies  one of the following: $(\rmn1)$ $Q$ is the Hasse quiver of an array TAOS; $(\rmn2)$  all algebraically connected components of $Q$ are Hasse quivers of some array TAOSs, and 
for any two distinct algebraically connected components \(Q^{i}\) and
\(Q^{j}\) of \(Q\), there do not exist paths of length at least two \(p_i\) in
\(Q^{i}\) and \(p_j\) in \(Q^{j}\) such that $s(p_i)=s(p_j)$, $t(p_i)=t(p_j)$.
 Let $Q'$ be the quiver $Q$ with some extra multiple arrows, i.e., $Q'_0 = Q_0$, $Q'_1 = Q_1 \sqcup E$, and for any arrow $\alpha_1 \in E$, there exists an arrow $\alpha_2 \in Q_1$ such that $s(\alpha_1) = s(\alpha_2)$ and $t(\alpha_1) = t(\alpha_2)$. Then there exists an inner product $\langle\cdot,\cdot\rangle$ on $\n_{Q'} / I$ such that $(\n_{Q'} / I, \langle\cdot,\cdot\rangle)$ is an algebraic Ricci soliton with respect to $-1$, where $I$ is the ideal of $\n_{Q'}$ generated by all differences $\omega_1 - \omega_2$ of two paths $\omega_1$ and $\omega_2$ of length at least two that start at the same vertex and end at the same vertex.
\end{thm}

\begin{proof}
For each pair of vertices $u, v \in Q'_0$, if there is exactly one arrow $x \in Q_1'$ from $u$ to $v$, keep $x$ as is. If there are $k$ arrows $a_1, \dots, a_k \in Q_1'$ from $u$ to $v$ with $k \ge 2$, replace $a_1, \dots, a_k$ by the following $k$ elements of $\n_{Q'} / I$:
\[
\frac{1}{k} \sum_{i=1}^k a_i,\qquad
\frac{k-1}{k} a_j + \Bigl(-\frac{1}{k}\Bigr) \sum_{i \neq j} a_i,\quad j = 1, \dots, k-1,
\]
denoted by $a, a_1', \dots, a_{k-1}'$, respectively. When $k = 1$, the element $a$ is simply the unique arrow. It is easy to check that $[a_j', \n_{Q'} / I] = 0$ for $j = 1, \dots, k-1$, and that the vector space spanned by $a_1, \dots, a_k$ coincides with that spanned by $a, a_1', \dots, a_{k-1}'$.

The arrows in $Q_1'$ together with the equivalence classes of paths of length greater than one form a basis of $\n_{Q'} / I$, but this basis is not necessarily nice. However, applying the above replacement yields a nice basis $\beta$. Denote the set of equivalence classes of paths of length greater than one by $\beta_1$, and partition $\beta \setminus \beta_1$ into
\begin{eqnarray*}
\beta_2 &=& \{\text{elements in } \beta \setminus \beta_1 \text{ with all coefficients } \geq 0\},\\
\beta_3 &=& \{\text{elements in } \beta \setminus \beta_1 \text{ with some coefficients }   <     0\}.
\end{eqnarray*}
Elements in $\beta_2$ are of the form $a$ above, while those in $\beta_3$ are of the form $a_1', \dots, a_{k-1}'$. Let $\n_1$ be the ideal generated by $\beta_1 \cup \beta_2$, and $\n_2$ the ideal generated by $\beta_3$. Since $[a_j', \n_{Q'} / I] = 0$ for each $j$, we have that $\n_2$ is abelian and $[\n_1, \n_2] = 0$. Moreover, as vector spaces, $\n_1$ is spanned by $\beta_1 \cup \beta_2$ and $\n_2$ is spanned by $\beta_3$. Hence $\n_{Q'} / I = \n_1 \oplus \n_2$ as Lie algebras.

We now show $\n_1 \cong R(P)$. The set $\beta_1 \cup \beta_2$ is a basis of $\n_1$. Define a linear map $f: \n_1 \to R(P)$ as follows. For $x \in \beta_1$, set $f(x) = (s,t)$, where $s$ and $t$ are the starting and terminal vertices of $x$. For $x \in \beta_2$, write $x = \frac{1}{k} \sum_{i=1}^k x_i$ with $x_1,\dots,x_k \in Q_1'$ and $s(x_1) = \cdots = s(x_k)$, $t(x_1) = \cdots = t(x_k)$; set $f(x) = (s(x_1), t(x_1))$. One checks that $f$ is bijective and respects the Lie bracket, hence $f$ is an isomorphism. Thus $\n_1 \cong R(P)$ as Lie algebras. Consequently, there exists an inner product $\langle\cdot,\cdot\rangle_1$ on $\n_1$ such that $(\n_1, \langle\cdot,\cdot\rangle_1)$ is an algebraic Ricci soliton with respect to $-1$, and $\n_2$ is abelian. By Proposition~\ref{prop:abelian direct sum}, the theorem follows.
\end{proof}

\bibliographystyle{plain}
\bibliography{Riccisoliton}

\end{document}